\documentclass{amsart}
\usepackage{hyperref,amssymb,xcolor}

\newtheorem{theorem}{Theorem}[section]
\newtheorem{lemma}[theorem]{Lemma}

\theoremstyle{definition}

\theoremstyle{remark}

\numberwithin{equation}{section}

\begin{document}

\title{Completeness of Discrete Translates in $H^1(\mathbb{R})$}

\author{Bhawna Dharra}
\address{Department of Mathematics, Indian Institute of Technology Delhi, 
New Delhi-110016, India}
\email{bhawna@maths.iitd.ac.in}

\author{S. Sivananthan }
\address{Department of Mathematics, Indian Institute of Technology Delhi, 
New Delhi-110016, India}
\email{siva@maths.iitd.ac.in}




\begin{abstract}
We provide a characterization of discrete sets $\Lambda \subset \mathbb{R}$ that admit a function whose $\Lambda$-translates are complete in the Hardy space $H^1(\mathbb{R})$. In particular, we show that such a set cannot be uniformly discrete. We then give a uniformly discrete $\Lambda \subset \mathbb{R}$ which admits a pair of functions such that their $\Lambda$-translates are complete in $H^1(\mathbb{R})$.
\end{abstract}

\maketitle


\section{Introduction}
 The Hardy space plays an important role in the theory of Fourier analysis. The theory of Hardy space originated to understand the convergence of the Fourier series and the growth of analytic functions when it approaches the boundary of the unit disk in the complex plane \cite{duren1970, Zygmund1, Zygmund2}. The seminal works of Stein and Weiss \cite{steinweiss60}, and Fefferman and Stein \cite{feffermanstein72} developed the real Hardy spaces  on $n$-dimensional space $H^p(\mathbb{R}^n)$ $(0<p<\infty)$. They characterized the space with different definitions that allow working with the space easily. In fact, for $1<p<\infty$, $H^p$ coincides with the Lebesgue spaces $L^p$. The fundamental operators in Harmonic Analysis, like singular integral operators, maximal operators, and Littlewood-Paley functions, are well behaved (bounded) on the Lebesgue spaces $L^p(\mathbb{R}^n)$, $1<p<\infty$ but are not bounded on $L^1(\mathbb{R}^n)$. The Hardy space $H^1(\mathbb{R}^n)$ appears as a natural substitute in order to overcome the shortcomings of $L^1(\mathbb{R}^n)$. Many results concerning the completeness of translates of a function are studied only for the classical Lebesgue spaces \cite{atzmonolevskii96, beurling51, blank06, olevskiiulanovskii06, levolevskii11, nikolski99, olevskiiulanovskii18I, olevskiiulanovskii18II}. We thus aim to investigate the completeness of systems formed by discrete translates of a function in the Hardy space $H^1(\mathbb{R}^n)$.

Wiener \cite{wiener32} characterized the functions $\phi$ whose all translates are complete in the
Lebesgue spaces $L^1(\mathbb{R})$ and $L^2(\mathbb{R})$ in terms of the zero set of its Fourier transform. To be precise, Wiener proved that $span \{\tau_{\lambda}\phi\}_{\lambda \in \mathbb{R} }$ is dense in $L^1(\mathbb{R})$ if and only if its
Fourier transform $\hat{\phi}$ never vanishes, and $span \{\tau_{\lambda}\phi \}_{\lambda \in \mathbb{R} }$ is dense in $L^2(\mathbb{R})$ if and only if $\hat{\phi} \neq 0 $ $a.e.$ on $\mathbb{R}$, where we define $\tau_{\lambda}\phi(\cdot):= \phi(\cdot - \lambda) $. For $1<p<2$, Beurling \cite{beurling51} gave a sufficient condition that $ \{\tau_{\lambda}\phi \}_{\lambda \in \mathbb{R}}$ is complete in $L^p(\mathbb{R})$ whenever Hausdorff dimension of the zero set of $\hat{\phi}$ is less than ${2(p-1)}/{p}$. However, Lev and Olevskii \cite{levolevskii11} proved that a characterization in terms of zero set of Fourier transform is not possible for $L^p(\mathbb{R})$ spaces when $1<p<2$. More generally, Feichtinger and Gumber \cite{feichtingergumber20} showed that $\hat{g}(x) \neq 0$ for all $x \in \mathbb{R}^n$ is sufficient for all translates to be complete in any minimal tempered standard space which are defined to be translation and modulation invariant distribution spaces satisfying some conditions.

We are particularly interested in system of discrete translates, i.e., when $\Lambda$ is a discrete subset of  $\mathbb{R}^n$. It easily follows from the Plancherral theorem that there does not exist any function $\Phi$ whose $\mathbb{Z}^n$-translates are complete in $L^2(\mathbb{R}^n)$. Same holds true for $L^p(\mathbb{R}^n)$ for $1 \leq p < 2$ \cite{nikolski99}. However, for $p > 2$ there exist functions in $L^p(\mathbb{R}^n)$ whose $\mathbb{Z}^n_{+}$-translates are complete (\cite{atzmonolevskii96}, \cite{nikolski99}). The situation changes when we consider uniformly discrete sets other than $\mathbb{Z}^n$. Olevskii and Ulanovskii \cite{olevskiiulanovskii18I} showed that there exist uniformly discrete set $\Lambda$ and a Schwartz class function whose $\Lambda$-translates are complete in $L^p(\mathbb{R}^n)$ for any $p > 1$. 

 A complete characterization of all the sets $\Lambda$ admitting functions whose $\Lambda$-translates are complete was first given for $L^1(\mathbb{R})$ \cite{olevskiiulanovskii06}. Later, Blank \cite{blank06} extended the results for the weighted Lebesgue spaces $L^1_w(\mathbb{R})$, where $w$ is a non-quasianalytic weight function.

$H^1(\mathbb{R})$ and $L^1(\mathbb{R})$ are Banach algebras that have contrasting behaviours even though $H^1(\mathbb{R}) \subset L^1(\mathbb{R})$. For example, $L^1(\mathbb{R})$ possess bounded approximate identity but  $H^1(\mathbb{R})$ has no bounded approximate identity. Under modulation operator, the space $L^1(\mathbb{R})$ is invariant  but $H^1(\mathbb{R})$ is not invariant (See \cite{johnsonwarner10}). Despite the disparity in these spaces, we get an analogous result to $L^1(\mathbb{R})$ (\cite{olevskiiulanovskii06}) that for any discrete $\Lambda \subset \mathbb{R}$, $\Lambda$-translates of some function in $H^1(\mathbb{R})$ are complete if and only if $\Lambda$ is large in a certain sense:
\begin{theorem}\label{generatingsetch}
Let $\Lambda \subset \mathbb{R}$ be a discrete set. Then there exists $f \in H^1(\mathbb{R})$ such that $ \{\tau_{\lambda}f\}_{\lambda \in \Lambda}$ is  complete in $ H^1(\mathbb{R})$ if and only if the Beurling-Malliavin density of $\Lambda$ $D_{BM}(\Lambda)$ is not finite.\\
\end{theorem}
Following \cite{olevskiiulanovskii16}, we define the (exterior) Beurling-Malliavin Density $D_{BM}(\Lambda)$ for any discrete set $\Lambda \subset \mathbb{R}$ as the supremum of all the numbers $D>0$ such that there is a substantial sequence of intervals $\{I_k\}$ satisfying $$\#(\Lambda \cap I_k) \geq D|I_k|, \hspace{0.2in} k \in \mathbb{N}, \quad \forall k\in \mathbb{N}.$$
Here, a sequence of intervals $I_k, k \in \mathbb{N}$, which belong to the positive or negative half-axis, is called substantial if the intervals are disjoint, $|I_k|>1$ for all $ k \in \mathbb{N}$, and
$$\sum\limits_{k=1}^{\infty}\left( \frac{|I_k|}{\text{dist}(I_k,0)}  \right)^2 = \infty.$$

As a  consequence, uniformly discrete translates of any single function can never be complete in both $H^1(\mathbb{R})$ and $L^1(\mathbb{R})$. However, similar to $L^1(\mathbb{R})$ (\cite{olevskiiulanovskii06}), there exists a pair of functions in $H^1(\mathbb{R})$ such that their translates from a small perturbation of integers are complete:
\begin{theorem}\label{2generator}
Let $\Lambda \subset \mathbb{R}$ be an exponentially small perturbation of integers, i.e., $\Lambda = \{ n+r_n\}_{n \in \mathbb{Z}}$ where
\begin{equation}\label{perturbedintegers}
\begin{split}
    \Lambda = \{ n+r_n\}_{n \in \mathbb{Z}} \hspace{1.8in}\\
    0 < |r_n| \leq \gamma^{|n|} \text{ for all } n \in \mathbb{Z} \text{ and some } 0 < \gamma < 1. 
\end{split}
\end{equation}
Then there exists a pair of functions $f_1$ and $f_2$ in $H^1(\mathbb{R})$ such that $\text{span}\{ \tau_{\lambda}f_1, \tau_{\nu}f_2\}_{\lambda, \nu \in \Lambda}$ is dense in $H^1(\mathbb{R})$.  
\end{theorem}

The organization of this note is as follows: in \autoref{preliminaries} we give some preliminary results.  In \autoref{AnalogueOfWTTForH1}, we provide a characterization of functions whose all translates are complete in $H^1(\mathbb{R})$. In \autoref{CharacterizationOfGeneratingSets} we prove \autoref{generatingsetch}, and in \autoref{APairOfGenerators}, we prove \autoref{2generator}. 

\section{Preliminaries}\label{preliminaries}

We define the Hardy space $H^1(\mathbb{R})$ of functions on $\mathbb{R}$ as:
$$ H^1(\mathbb{R}) := \{f \in L^1(\mathbb{R}) \text{ }|\text{ } Hf \in L^1(\mathbb{R})\} $$
$$
\text{ with } \|f\|_{H^1}:=\|f\|_{L^1}+\|Hf\|_{L^1}
$$
where $Hf$ denote the Hilbert transform of $f$ which is a multiplier defined by 
$$\widehat{Hf}(\zeta)   = -i\frac{\zeta}{|\zeta|} \hat{f}(\zeta) \text{ and } $$
$$\hat{f}(\zeta)=\int\limits_{\mathbb{R}}f(x)e^{ix\zeta}dx.$$
It follows that $H^1(\mathbb{R})$ is a Banach algebra under convolution. Moreover, $L^1(\mathbb{R})*H^1(\mathbb{R}) \subseteq H^1(\mathbb{R})$.\\

$H^1(\mathbb{R})$ has an approximate identity, i.e., there exists $\{v_n\}_{n \in \mathbb{N}} \subset H^1(\mathbb{R})$ such that $v_n * f \xrightarrow{\| \cdot \|_{H^1}} f$ as $n \to \infty$ for all $f \in H^1(\mathbb{R})$. However, these are unbounded as shown in \cite{johnsonwarner10}.\\

We can decompose elements of $H^1(\mathbb{R})$ into an $\ell^1$ sum of  molecules. Suppose that $1 < q \leq \infty$, $a_0>0$ and $b_0=1-\frac{1}{q}+a_0$. A function $m \in L^q(\mathbb{R})$ is said to be a $(q,a_0)$-molecule centered at $x_0 \in \mathbb{R}$ if it satisfies
\begin{enumerate}
    \item $|x|^{b_0}m(x) \in L^q(\mathbb{R})$,
    \item $\mathcal{N}_q(m):=\|m\|_q^{a_0/b_0}\|m(\cdot)|\cdot-x_0|^{b_0}\|_q^{1-(a_0/b_0)} < \infty$ and,
    \item $\int\limits_{\mathbb{R}} m(x)dx=0$.
\end{enumerate}
The number $\mathcal{N}_q(m)$ is called the molecular norm of $m$. The molecular characterization of $H^1(\mathbb{R})$ is given by the following result:
\begin{theorem}\label{molecularcharH1}\cite[p.~84--85]{shanzhen95}
Let $1 < q \leq \infty$ and $a_0>0$. Then every $(q,a_0)$-molecule belongs to $H^1(\mathbb{R})$ and $\|m\|_{H^1} \leq C\mathcal{N}_q(m)$ where $C$ is a constant independent of $m$. A function $f$ belongs to $H^1(\mathbb{R})$ if and only if $f$ has a decomposition of the form 
$$f=\sum\limits_{k=1}^{\infty} c_km_k,  $$ 
where each $m_k$ is a $(q,a_0)$-molecule with uniformly bounded molecular norm, $\sum\limits_{k=1}^{\infty}|c_k|< \infty$, and the series converges in the tempered distribution topology. Moreover,
$$\inf\left\{ \sum\limits_{k=1}^{\infty}|c_k| : f=\sum\limits_{k=1}^{\infty} c_km_k \right\} \text{ is equivalent to } \|f\|_{H^1}.$$
\end{theorem}

A simple consequence is that $L^q(\mathbb{R}) \cap H^1(\mathbb{R})$ is dense in $H^1(\mathbb{R})$ for any $1<q \leq \infty$. It is also easy to verify that the set of all $f \in L^1(\mathbb{R})$ such that $\hat{f}$ has compact support away from $0$ is dense in $H^1(\mathbb{R})$.\\

The dual space of $H^1(\mathbb{R})$ is identified with the space of functions of bounded mean oscillations BMO$(\mathbb{R})$ defined by:
$$\text{BMO}(\mathbb{R})=\left\{  f\in L^1_{\text{loc}}(\mathbb{R}) :  \sup\limits_{I : I \text{ is an interval}}\frac{1}{|I|}\int\limits_I|f(y)-f_I|dy  < \infty  \right\},$$ where $|I|$ denote the length of $I$ and $\displaystyle f_I = \frac{1}{|I|}\int\limits_I f(x) ~dx$, {with }\\
 $$\|f\|_{\text{BMO}}:=\sup\limits_{I : I \text{ is an interval}}\frac{1}{|I|}\int\limits_I|f(y)-f_I|dy.$$

Note that $\|f\|_{\text{BMO}}$ is a semi-norm with $\|f\|_{\text{BMO}}=0$ if and only if $f$ is a constant.\\
Clearly,  $L^{\infty}(\mathbb{R}) \subset \text{ BMO}(\mathbb{R})$. In fact, $L^{\infty}(\mathbb{R})$ is a proper subspace of BMO$(\mathbb{R})$. For example, the function $\log|x|$ is in $ \text{BMO}(\mathbb{R})$ but not in $L^{\infty}(\mathbb{R})$. However, we have \cite[p.~25]{uchiyama01}
\begin{equation}\label{bmogrowth}
    \int\limits_{\mathbb{R}}\frac{|f(x)|}{(1+|x|)^{2}}dx< \infty \text{ for all } f \in \text{BMO}(\mathbb{R}),
\end{equation}
which shows BMO$(\mathbb{R})$ is contained in the space of tempered distributions $\mathcal{S}^{'}(\mathbb{R})$.

We can also approximate the dual action of any function in BMO$(\mathbb{R})$ by $L^{\infty}$ functions, i.e., $L^{\infty}(\mathbb{R})$ is weak* dense in BMO$(\mathbb{R})$:
\begin{theorem}\label{bmodualapprox}\cite[p.~34]{uchiyama01}
For $h \in $ BMO$(\mathbb{R})$ and $r>0$, define
$$h_r(x)=h(x)/\max \{1,|h(x)|/r\}.$$
Then for any $f \in H^1(\mathbb{R})$, we get
\begin{equation}
    \begin{split}
        \|h_r\|_{L^{\infty}} \leq & r\\
        \|h_r\|_{\text{BMO}} \leq & \|h\|_{\text{BMO}}\\
        \langle f,h \rangle_{\text{BMO}} = & \lim\limits_{r \rightarrow \infty} \int\limits_{\mathbb{R}} f(x) h_r(x) dx.
    \end{split}
\end{equation}
\end{theorem} 

We refer to \cite{shanzhen95} and \cite{uchiyama01} for the general theory related to the Hardy spaces.

\section{Analogue of Wiener-Tauberian Theorem for $H^1(\mathbb{R})$}\label{AnalogueOfWTTForH1}

Johnson and Warner \cite{johnsonwarner10} identified the maximal ideal space $\Delta$ (all continuous homomorphisms
of $H^1(\mathbb{R})$ onto the complex numbers) of $H^1(\mathbb{R})$ with $\mathbb{R}\setminus \{0\}$. Further, they mentioned that all translates of a function in $H^1(\mathbb{R})$ are complete if and only if its Fourier transform never vanishes on $\mathbb{R} \backslash \{0\}$. However, they  deferred  the proof of the statement to another paper which we could not find in the literature. For the sake of exposition, we shall provide a proof in this note. Our proof uses Banach Algebra techniques whose basic terminology can be found in \cite{loomis1953}. \\

The following result is proved in \cite{johnsonwarner10}:\\

\begin{lemma}\label{maximalideal}\cite{johnsonwarner10}
The maximal ideal space $\Delta$ of $H^1(\mathbb{R})$ can be identified with $\mathbb{R} \backslash \{0\}$ and the Gelfand transform is equivalent to the Fourier transform restricted to $\mathbb{R} \backslash \{0\}$.
\end{lemma}

Using this, we can derive the following:
\begin{theorem}
The space $H^1(\mathbb{R})$ is a regular semi-simple Banach algebra.
\end{theorem}
\begin{proof}
By the uniqueness of Fourier transform on $L^1(\mathbb{R})$, $H^1(\mathbb{R})$ is semi-simple. In order to prove that $H^1(\mathbb{R})$ is regular, we need to show that its Gelfand representation is a regular function algebra. This is equivalent to say that for every (weakly) closed set $C \subset \Delta \equiv \mathbb{R} \backslash \{0\}$ and every point $p \notin C$, there exists  $f \in H^1(\mathbb{R})$ such that $ \hat{f} \equiv 0$ on $C$ and $\hat{f}(p) \neq 0$. But the weak topology on $\Delta$ defined by functions in the Gelfand representation of $H^1(\mathbb{R})$ is equivalent to the topology on $\mathbb{R} \backslash \{0\}$ generated by the Euclidean distance $d(\zeta,\eta)=|\zeta-\eta|$, which follows from the fact that for every function $f \in H^1(\mathbb{R})$, $\hat{f}$ is a bounded and continuous function with $\hat{f}(0)=0$. Now by Urysohn's lemma, there exists a continuous function $F$ having compact support containing $p$ away from $\{0\}$ and $C$. This $F=\hat{f}$ for some $f \in H^1(\mathbb{R})$ serves the purpose.
\end{proof}

Now the above results ensures that:
\begin{theorem}\label{properideal}
Every proper closed ideal of $H^1(\mathbb{R})$ is contained in a regular maximal ideal.
\end{theorem}
For the proof, we need the following Lemma.
\begin{lemma}\cite[p.~85]{loomis1953}
Let $B$ be a regular semi-simple Banach algebra with the property that the set of elements $f$ such that $\hat{f}$ has compact support is dense in $B$. Then every proper closed ideal is included in a regular maximal ideal. 
\end{lemma}
\begin{proof}[Proof of \autoref{properideal}]
Note that the Gelfand transform on $H^1(\mathbb{R})$ is equivalent to the Fourier transform restricted to $\mathbb{R} \backslash \{0\}$ and the set of all $f \in L^1(\mathbb{R})$ such that $\hat{f}$ has compact support away from $0$ is dense in $H^1(\mathbb{R})$. Thus $H^1(\mathbb{R})$ is a regular semi-simple Banach algebra such that the set of those functions whose Gelfand transform has compact support is dense in $ H^1(\mathbb{R})$. Hence the result follows from the above lemma.
\end{proof}

Next we have the following:
\begin{lemma}\label{idealinvariant}
Let $I$ be a closed linear subspace of $H^1(\mathbb{R})$. Then $I$ is translation-invariant if and only if $I$ is an ideal.
\end{lemma}
\begin{proof}
First, suppose that $I$ is an ideal. Let $f \in I$, $x \in \mathbb{R}$ and $\{v_n\}_{n \in \mathbb{N}}$ be an approximate identity sequence in $H^1(\mathbb{R})$. Then we have
\begin{align*}
    \tau_xf & = \tau_x \left( \lim\limits_{n \to \infty} v_n*f \right)\\
    & = \lim\limits_{n \to \infty} \tau_x(v_n*f) \\
    & = \lim\limits_{n \to \infty} (\tau_xv_n)*f
\end{align*}
Since $I$ is an ideal, $(\tau_xv_n)*f \in I$ for all $n \in \mathbb{N}$. Moreover, since $I$ is closed, their limit also belongs to $I$.\\
Now, suppose that $I$ is translation-invariant and define $I^{\perp}:=\{ h \in \text{BMO}(\mathbb{R}):\langle f,h \rangle_{\text{BMO}} =0 \text{ for all }f \in I  \}$. Let  $f \in I$, $g \in H^1(\mathbb{R})$. By the Hahn-Banach theorem, $f*g \in I$ if and only if $\langle  f*g,h \rangle_{\text{BMO}} =0$ for all $h \in I^{\perp}$. Indeed, for $h \in I^{\perp}$ using \autoref{bmodualapprox} we have
\begin{align*}
    \langle  f*g,h \rangle_{\text{BMO}} & = \lim\limits_{r \to \infty} \int\limits_{\mathbb{R}} f*g(x)h_r(x)dx\\
    & = \lim\limits_{r \to \infty} \int\limits_{\mathbb{R}} \int\limits_{\mathbb{R}} g(y)f(x-y)h_r(x)dydx\\
     & = \lim\limits_{r \to \infty}\int\limits_{\mathbb{R}}\left(\int\limits_{\mathbb{R}} f(x-y)h_r(x)dx\right)g(y)dy \\
     & = 0 \text{ as } h \in I^{\perp}, f \in I,
\end{align*}
where the penultimate step uses Fubini's theorem and the last step is consequent upon the Lebesgue dominated convergence theorem.
\end{proof}

Now from the above results, an analogue of Wiener-Tauberian theorem for $H^1(\mathbb{R})$ can be deduced as follows:\\
\begin{theorem}\label{WTTH1}
Let $f \in H^1(\mathbb{R})$. Then  $\{\tau_{\lambda}f\}_{\lambda \in \mathbb{R}}$ is complete in $H^1(\mathbb{R})$ if and only if its Fourier transform never vanishes on $\mathbb{R} \backslash \{0\}$.
\end{theorem}
\begin{proof}
Let $f \in H^1(\mathbb{R})$. Then $I_f=\overline{\text{span}\{ \tau_{\lambda}f \}}_{\lambda \in \mathbb{R}}$ is a closed ideal by Lemma \ref{idealinvariant}. First suppose that $\hat{f}$ never vanishes on $\mathbb{R} \backslash \{0\}$. If $I_f$ is proper, then by \autoref{properideal}, it is contained in a regular maximal ideal, i.e., $\hat{f}(\zeta)=0$ for some $\zeta \in \mathbb{R} \backslash \{0\}$, a contradiction. Thus $ I_f=H^1(\mathbb{R})$. Now suppose that $\hat{f}(\zeta_0)=0$ for some $\zeta_0 \in \mathbb{R} \backslash \{0\}$. Then $I_f$ will be contained in the maximal ideal corresponding to $\zeta_0$ and thus $ I_f \neq H^1(\mathbb{R})$.
\end{proof}

It is worth mentioning that the result is valid for $H^1(\mathbb{R}^n)$ for general $n \in \mathbb{N}$ and its proof follows on the same line as above. In particular, this implies that for $\Lambda$-translates of a function $f$ to be complete in $H^1(\mathbb{R}^n)$ for any $\Lambda \subseteq \mathbb{R}^n$, it is necessary that $\hat{f}(\zeta)$ is non-zero for every $\zeta \in \mathbb{R}^n \backslash \{0\}$.
\section{Discrete Translates Of A Single Function}\label{CharacterizationOfGeneratingSets}

In this section, we will prove \autoref{generatingsetch}. The heart of the proof lies in a consequence of the well-known Beurling-Malliavin Density theorem \cite{beurlingmalliavin62, beurlingmalliavin67} according to which for any discrete set $\Lambda \subset \mathbb{R}$, 
$$\pi D_{BM}(\Lambda) = R(\Lambda).$$
Here, $R(\Lambda)$ is the spectral radius of $\Lambda$ defined by 
$$R(\Lambda):= \sup \{ r>0 \text{ such that } E(\Lambda) \text{ is dense in } L^2(-r,r) \}$$
where $$E(\Lambda):= span \{ e^{i \lambda \cdot} : \lambda \in \Lambda \}$$ is the set of all trigonometric polynomials with frequencies from $\Lambda$.\\

The Beurling-Malliavin Density theorem ensures, for any function in $H^1(\mathbb{R})$, the existence of a non-trivial function $f \in \text{BMO}(\mathbb{R})$ annihilating all $\Lambda$-translates of it whenever $D_{BM}(\Lambda)$ is finite and the existence of a function whose $\Lambda$-translates are complete in $H^1(\mathbb{R})$ whenever $D_{BM}(\Lambda)$ is infinite.\\

\begin{proof}[Proof of \autoref{generatingsetch}]
First, suppose that $D_{BM}(\Lambda)$ is finite. By Beurling-Malliavin Density theorem, there exists a non-trivial $g \in L^2(\mathbb{R})$ such that $g(\lambda)=0$ for all $\lambda \in \Lambda$ and $supp (\hat{g})$ is compact. We replace $g$ by a modulation of it so that $0 \notin supp(\hat{g})$. Now let $f \in H^1(\mathbb{R})$. In light of \autoref{WTTH1}, we may assume that $\hat{f}$ never vanishes on $\mathbb{R} \backslash \{0\}$.
Define $K$ by
\[
K(\zeta) = \begin{cases} 
      \frac{\hat{g}(\zeta)}{\hat{f}(\zeta)} & \text{ for }\zeta \in supp(\hat{g}), \\
      0 & \text{ otherwise} .
   \end{cases}
\]
 Then $K$ is a non-trivial continuous function with compact support. So, $k=\check{K} \in L^{\infty}(\mathbb{R}) \subset \text{BMO}(\mathbb{R})$ such that for any $\lambda \in \Lambda$
\begin{align*}
    \langle k, \tau_{\lambda}f \rangle_{\text{BMO}} & = \int\limits_{\mathbb{R}}\hat{k}(\zeta) \hat{f}(\zeta) e^{i \lambda \zeta} d \zeta\\
    & = \int\limits_{supp(\hat{g})}\hat{g}(\zeta)e^{i \lambda \zeta} d \zeta \\
    & = \sqrt{2 \pi} g(\lambda)\\
    & = 0,
\end{align*}
where the third step follows from the Fourier-Inversion Formula. This completes the proof of necessity part. 

Now suppose that $D_{BM}(\Lambda)$ is infinite. A simple consequence of the Beurling-Malliavin Density theorem is that $E(\Lambda)$ is dense in $W(I)$ for any bounded interval $I \subset \mathbb{R}$, where $W(I)$ is the Sobolev space defined by $W(I):=\{ F \in L^2(I) : F^{'} \in L^2(I) \}$ equipped with $\| F \|_{W(I)}=\| F \|_{L^2(I)} + \| F^{'} \|_{L^2(I)}$ :
\begin{lemma}\cite[Lemma ~12.16]{olevskiiulanovskii16}
Let $I$ be a bounded interval. Then every $F \in W(I)$ admits approximation in $W(I)$ with an arbitrarily small error by a polynomial
$$q(t) = \sum\limits_{\lambda \in \Lambda} c_{\lambda}e^{ i \lambda t}$$
\end{lemma}

Because of the density of the exponentials, it is natural to consider the space $$X:=\{ \hat{f} | f \in H^1(\mathbb{R}) \}$$
with $\| \hat{f} \|_X=\| f \|_{H^1(\mathbb{R})}$ for all $f \in H^1(\mathbb{R})$.\\
Moreover, we observe that a closed subspace of $W(\mathbb{R})$ is dense in $X$ with Sobolev norm stronger than $\|\cdot\|_X$ on it:

\begin{lemma}\label{W0densityX}
Define $$W_0(\mathbb{R}):=\{ F \in W(\mathbb{R}) | F(0)=0\}.$$
Then $W_0(\mathbb{R})$ is dense in $X$ with $\|F\|_X \leq d_0 \|F\|_{W(\mathbb{R})}$ for all $F \in W_0(\mathbb{R})$ where $d_0$ is a constant independent of $F$.
\end{lemma}
\begin{proof}
By the virtue of the definition, $\check{F}$ is a $(2,\frac{1}{2})$-molecule for each $F \in W_0(\mathbb{R})$. The lemma now follows from the molecular characterization of $H^1(\mathbb{R})$ in \autoref{molecularcharH1} and
\[
\begin{array}{rl}
    \|F\|_X & = \| \check{F}\|_{H^1(\mathbb{R})}  \\
     & \leq  C_1 \text{ molecular norm of }\check{F}   \\
     & =  C_1 (\|\check{F}\|_{2}\|x\check{F}\|_{2})^{\frac{1}{2}}\\
     & \leq  C_2 (\|F\|_2\|F^{'}\|_2)^{\frac{1}{2}}\\
     & \leq  C_2 \frac{\|F\|_2+\|F^{'}\|_2}{2}\\
     & = d_0 \|F\|_{W(\mathbb{R})},
\end{array}
\]
where the third and fourth step follows from Plancherel's theorem and AM-GM Inequality respectively.
\end{proof}

%

From the above results, one can construct a function whose $\Lambda$-translates are complete in $H^1(\mathbb{R})$ analogous to Lemma 12.17 in \cite{olevskiiulanovskii16} with appropriate modifications around $0$. 
For completeness, we will give an example of one such function $f$ whose Fourier transform $\hat{f}=F$ is an infinite combination of integer translates of the tent map $\Phi$ defined by
\[
\Phi(\zeta) = \begin{cases} 
      1+\zeta & -1 \leq \zeta \leq 0, \\
      1-\zeta & 0 \leq \zeta \leq 1 ,\\
      0 & \text{ otherwise} .
   \end{cases}
\]
We can choose a dense countable set $ \mathcal{G}:=\{G_n\}_{n \in \mathbb{N}}$ in $W_0(\mathbb{R})$ such that $supp(G_n) \subseteq I_n:=[-n,-\frac{1}{n}] \cup [\frac{1}{n},n]$ and a sequence of distinct positive real numbers $\{ \epsilon_n \}_{n \in \mathbb{N}}$ such that $\epsilon_n \leq \frac{1}{2}$ and $\epsilon_n \downarrow 0$. We will inductively define a sequence of non-negative numbers $\{\delta_i\}_{i \in \mathbb{Z}}$ and $\{p_n\}_{n \in \mathbb{N}} \subseteq E(\Lambda) $ such that
$$F=\sum\limits_{i \in \mathbb{Z}} \delta_i\tau_i \Phi \in W_0(\mathbb{R}) \text{ and}$$
$$\|G_n-p_nF\|_{W(\mathbb{R})}<\epsilon_n \text{ for all } n \in \mathbb{N}.$$

For any absolutely continuous function $K$ on $\mathbb{R}$, denote $\|K\|_{*}=\|K\|_{\infty}+\|K^{'}\|_{\infty}$. Then one can easily see that for any $G \in W(\mathbb{R})$, $\|KG\|_{W(\mathbb{R})} \leq \|K\|_{*} \|G\|_{W(\mathbb{R})}$.\\

Note that if we define for $n \in \mathbb{N}$,
$$F_n:=\sum\limits_{i = -n}^n \delta_i\tau_i \Phi,$$
then one can easily observe that
\begin{enumerate}
    \item $supp(F_n)=[-(n+1),n+1]$,
    \item $\|F_n-F_{n-1}\|_{W(\mathbb{R})} < 4 \delta_n$,
    \item $\|F_n\|_{*} = 2 \max\limits_{i=1,2, \ldots,n} \delta_i $ and 
    \item $\|F_n\|_{W(\mathbb{R})}\leq C \sum\limits_{i=-n}^n \delta_i$ for a constant $C$ independent of $n$.
\end{enumerate}

Now we set $\delta_0=0$, $\delta_1=\delta_{-1}=\epsilon_1-\epsilon_2$ and define $\tilde{\delta_n}=\epsilon_n-\epsilon_{n+1}$ for all $n \in \mathbb{N}$.\\

Since $G_1=0$, by density of $E(\Lambda)$ in $W([-2,2])$, there exists $p_1 \in E(\Lambda)$ such that 
 \[
 \left\|\frac{G_1}{F_1}-p_1\right\|_{W([-2,2])}=\|p_1\|_{W([-2,2])}<\tilde{\delta_1} .
 \]
 This implies
 \[
 \Rightarrow \|G_1 -p_1F_1\|_{W(\mathbb{R})} < \|F_1\|_{*}\tilde{\delta_1} <\tilde{\delta_1}.
 \]
 
Now assume $\delta_k$ and $p_k$ are defined for all $k<n$. \\
Set $\delta_n=\delta_{-n}=\frac{\tilde{\delta_n}}{4 \max\{1,\|p_1\|_{*}, \ldots, \|p_{n-1}\|_{*}\}}$.\\
Since $supp(G_n) \subseteq I_n$ and $F_n\geq \min\{ \delta_n, \frac{\delta_1}{n}\}>0$ on $I_n$, so $\frac{G_n}{F_n} \in W([-(n+1),n+1])$ and thus by density of $E(\Lambda)$ in $W([-(n+1),n+1])$, there exists $p_n \in E(\Lambda)$ such that 
 \[
 \left\|\frac{G_n}{F_n}-p_n\right\|_{W([-(n+1),n+1])}<\tilde{\delta_n} .
 \]
 This implies
 \[
 \|G_n -p_nF_n\|_{W(\mathbb{R})} < \|F_n\|_{*}\tilde{\delta_n} <\tilde{\delta_n}.
 \]
 
Thus
\begin{align*}
     \|G_n-p_nF\|_{W(\mathbb{R})} \leq & \|G_n-p_nF_n\|_{W(\mathbb{R})} + \|p_nF_n-p_nF\|_{W(\mathbb{R})} \\
     < & \tilde{\delta_n} + \|p_n\|_{*} \|F_n-F\|_{W(\mathbb{R})}\\
     < & \tilde{\delta_n} + \|p_n\|_{*} \sum\limits_{k=n+1}^{\infty}\|F_k-F_{k-1}\|_{W(\mathbb{R})}\\
     < & \tilde{\delta_n} + \sum\limits_{k=n+1}^{\infty}\tilde{\delta_k} =  \epsilon_n.
 \end{align*}

Now indeed $\Lambda$-translates of $f=\check{F}$ are complete in $H^1(\mathbb{R})$. To see this, let $g \in H^1(\mathbb{R})$ and $\epsilon >0$. Then there exists infinitely many $n$ such that $\|g-\check{G_n}\|_{H^1(\mathbb{R})}< \epsilon/2$. Choose $N \in \mathbb{N}$ large enough such that $\epsilon_{N}<\epsilon/2d_0$, where $d_0$ is as in \autoref{W0densityX}. Let $p_N=\sum\limits_{\lambda \in \Lambda_N} c_{\lambda}^N e^{i \lambda \zeta}$  where $\Lambda_N \subset \Lambda$ is finite. Then \\
\begin{align*}
    \|g-\sum\limits_{\lambda \in \Lambda_{N}}c_{\lambda}^{N}\tau_{\lambda}f \|_{H^1(\mathbb{R})} \leq & \|g-\check{G_N}\|_{H^1(\mathbb{R})} + \|\check{G_N} - \sum\limits_{\lambda \in \Lambda_{N}}c_{\lambda}^{N}\tau_{\lambda}f\|_{H^1(\mathbb{R})}\\
    < & \epsilon/2 + d_0 \|G_N-p_{N}F\|_{W(\mathbb{R})} \\
    < & \epsilon/2 + d_0 \epsilon_{N}\\
    < & \epsilon.
\end{align*}
This completes the proof.
\end{proof}

It follows from the above result that if we want certain discrete translates of any one function to be complete in $H^1(\mathbb{R})$, then the set itself must be quite large. So the set cannot be uniformly discrete, or more generally cannot be of finite density. Moreover, it is easy to note that when we take $\Lambda=\mathbb{Z}$, there does not exist any finite collection of functions such that their integer-translates are complete in $H^1(\mathbb{R})$. In contrast to this, when we consider certain perturbations of $\mathbb{Z}$, there do exist finitely many functions such that their translates are complete in $H^1(\mathbb{R})$ as shown in the next section.

\section{Uniformly Discrete Translates Of A Pair Of Functions}\label{APairOfGenerators}
In this section, we prove \autoref{2generator}, that is whenever we consider ``very small" perturbation of integers $\Lambda$, then there exist two functions whose $\Lambda$-translates are complete in $H^1(\mathbb{R})$. The essence of the proof lies in the result \cite[Proposition ~12.21]{olevskiiulanovskii16} that whenever $\Lambda$ satisfies \autoref{perturbedintegers}, then it is a uniqueness set for the class of entire functions $\phi$ satisfying
\begin{enumerate}
    \item[(i)] $\phi(\cdot +iy) \in L^{\infty}(\mathbb{R})$, for every $y \in \mathbb{R}$ and
    \item[(ii)] the spectrum of $\phi(x)$ lies in $I+2 \pi \mathbb{Z}$, where $I$, $|I|<2 \pi$, is a closed interval.
    \item[(iii)] $\phi|_{\Lambda}=0$
\end{enumerate}
That is, if any entire function $\phi$ satisfying above conditions, then $\phi\equiv 0$.\\

\begin{proof}[Proof of \autoref{2generator}]
Suppose for any closed interval $I$ with $|I|<2 \pi$, we can find $f \in H^1(\mathbb{R})$ satisfying
$$\phi(z) := \int\limits_{\mathbb{R}}h(x)f(x-z)dx$$ 
has an entire extension for any $h \in \text{BMO}(\mathbb{R})$ and satisfies (i) and (ii) above. Let $I_1$ and $I_2$ be two closed intervals such that $0<|I_1|,|I_2|<2 \pi$ and $[-\pi , \pi] \subset I_1 \cup I_2$ . Now choose $f_1$ and $f_2$  as above with $I=I_1$ and $I=I_2$, which serve the purpose.\\

Now fix a closed interval $I$ with $0<|I|<2 \pi$. It is sufficient to find an entire function $f=\hat{F} \in H^1(\mathbb{R})$ such that 
\begin{enumerate}
    \item $f(\cdot + iy) \in H^1(\mathbb{R})$ for all $y \in \mathbb{R}$ and 
    \item Zero set of $F$ is $I+2\pi \mathbb{Z} \cup \{0\}$ and $F > 0$ elsewhere.
\end{enumerate}
Then we have $\|\phi(\cdot+iy)\|_{\infty} \leq \|h\|_{\text{BMO}}\|f(\cdot-iy)\|_{H^1}$ for all $y \in \mathbb{R}$ and $spec(\phi) \subseteq spec(\tilde{f})=-I+2 \pi \mathbb{Z}$ where $\tilde{f}(x)=f(-x)$. So $\phi$ is an entire function satisfying (i) and (ii).\\

One such choice can be made as follows: Let $J$ be an open interval such that $(I+2\pi \mathbb{Z}) \cap J = \emptyset$ and $|I|+|J|=2 \pi$. Choose any smooth function $G$ with $G>0$ on $J$ and vanishes outside $J$. Define
\[
F(t) = \sum\limits_{n \in \mathbb{Z}} c_n G(t-2\pi n),
\]
where we choose $\{c_n\}$ to be sufficiently small, say $c_n=e^{-|n|}$. Then 
\[
|F(t)| \leq C_b e^{-b|t|} \text{ for all }t, \text{ and every } b>0.
\]
So $f=\check{F}$ is the restriction to $\mathbb{R}$ of an entire function. For $f$ to be in $H^1(\mathbb{R})$, we need to redefine $F$ around $0$. If $0 \notin I+2 \pi \mathbb{Z}$, then choose $n_0$ such that $0 \in J + 2 \pi n_0$ and replace $c_{n_0}G(t-2 \pi n_0)$ in the sum defining $F$ by $c_{n_0}\tilde{G}(t-2 \pi n_0)$ where $\tilde{G}$ is a smooth function such that $\tilde{G}>0$ on $J + 2 \pi n_0 \setminus \{ 0\}$. Also, we know that if $h$ belongs to the Schwartz space $\mathcal{S}(\mathbb{R})$ with $h(0)=0$, then $\check{h} \in H^1(\mathbb{R})$. Therefore, $f$ satisfies the required conditions.
\end{proof}

In conclusion, for a uniformly discrete set $\Lambda \subset \mathbb{R}$, the minimum number of functions required such that their $\Lambda$-translates span whole of $H^1(\mathbb{R})$ must be greater than $1$. It remains open whether we can obtain a finite collection of functions for $\Lambda$ other than those satisfying \autoref{perturbedintegers}.


\section*{Acknowledgement}
We are thankful to Prof. D. Venku Naidu and Prof. P. Viswanathan  for their  valuable suggestions that improved the presentation of the paper.

\bibliographystyle{amsplain}

\end{document}